\documentclass[10pt]{amsart}

\usepackage{tikz}
\usepackage{amsfonts}
\usepackage{graphicx}
\usepackage{amssymb}
\usepackage{amsmath}
\usepackage{amsthm}
\usepackage{color}
\usepackage{graphicx,mathrsfs,comment}
\usepackage{hyperref,url}
\usepackage{mathrsfs}
\usepackage[shortlabels]{enumitem}
\usepackage{esint}
\usepackage{fancyhdr}

\definecolor{plum}{rgb}{.5,0,1}

\theoremstyle{plain}
\newtheorem{theorem}{Theorem}

\newtheorem{lemma}{Lemma}

\newtheorem{prop}{Proposition}
\newtheorem{question}{Question}

\theoremstyle{definition}
\newtheorem{definition}{Definition}

\theoremstyle{remark}
\newtheorem*{remark}{Remark}

\newcommand{\dist}{\textup{dist}}

\newcommand{\R}{\mathbb{R}}

\newcommand{\B}{\mathbb{B}}
\newcommand{\T}{\mathbb{T}}
\newcommand{\td}{\tilde}

\newcommand{\e}{\epsilon}

\newcommand{\Tmax}{\T_\mathrm{max}}

\newcommand{\eps}{\epsilon}

\newcommand{\leapp}{\lessapprox}
\newcommand{\geapp}{\gtrapprox}
\newcommand{\lesim}{\lesssim}
\newcommand{\gesim}{\gtrsim}

\newcommand{\tde}{\tilde{\delta}}
\newcommand{\ts}{\tilde{s}}
\newcommand{\tT}{\tilde{\T}}

\newcommand{\tM}{\tilde{M}}
\newcommand{\tN}{\tilde{N}}
\newcommand{\tR}{\tilde{r}}

\newcommand{\tW}{\tilde{W}}

\newcommand{\de}{\delta}
\newcommand{\cP}{\mathcal{P}}
\newcommand{\cL}{\mathcal{L}}

\newcommand{\cQ}{\mathcal Q}

\newcommand{\wt}{\widetilde}

\newcommand{\U}{\mathbb{U}}

\makeatletter
\@namedef{subjclassname@2020}{%
  \textup{2020} Mathematics Subject Classification}
\makeatother

\title[An incidence estimate and a Furstenberg type estimate for tubes]{An incidence estimate and a Furstenberg type estimate for tubes in $\mathbb{R}^2$}
\date{}

\author{Yuqiu Fu}
\address{Department of Mathematics, MIT,
Cambridge, MA 02139}
\email{yuqiufu@mit.edu}

\author{Shengwen Gan}
\address{Department of Mathematics, MIT,
Cambridge, MA 02139}
\email{shengwen@mit.edu}

\author{Kevin Ren}
\address{Department of Mathematics, MIT,
Cambridge, MA 02139}
\email{kevinren@mit.edu}


\begin{document}
\maketitle

\begin{abstract}

We study the $\delta$-discretized Szemer\'edi-Trotter theorem and Furstenberg set problem.
We prove sharp estimates for both two problems assuming tubes satisfy some spacing condition. For both problems, we construct sharp examples  that share common features.
\end{abstract}


\section{Introduction}

\subsection{Incidence estimate}
To begin with, let us first recall the famous Szemer\'edi-Trotter theorem in incidence geometry. Suppose $\cL$ is a set of lines in the plane. For $r\ge 2$, let $P_r(\cL)$ denote the $r$-rich points of $\cL$ --- the set of points that lie in at least $r$ lines of $\cL$. The Szemer\'edi-Trotter theorem gives sharp bounds for $|P_r(\cL)|$:
$$ |P_r(\cL)|\lesim \frac{|\cL|^3}{r^3}+\frac{|\cL|}{r}. $$
There is also a dual version. Suppose $\cP$ is a set of points in the plane. For $r\ge 2$, let $L_r(\cP)$ denote the $r$-rich lines of $\cP$ --- the set of lines that contain at least $r$ points of $\cP$. We have:
$$ |L_r(\cP)|\lesim \frac{|\cP|^3}{r^3}+\frac{|\cP|}{r}. $$

A natural question is to replace the points by $\de$-balls (the balls of radius $\de$) and the lines by the $\de$-tubes (the tubes of dimensions $1\times\de$), and then ask the incidence estimate between these $\de$-balls and $\de$-tubes.

This question is considered in \cite{guth2019incidence}, assuming some spacing conditions on tubes. 
In our paper, we generalize the incidence estimates on the plane in \cite{guth2019incidence}. We will consider some more general spacing conditions.

To state our results, we need the following notions.

\begin{definition}[Essentially distinct balls and tubes]\label{distinct}
For a set of $\delta$-balls $\B$, we say these balls are essentially distinct if for any $B_1\neq B_2\in\B$, $m(B_1\cap B_2)\leq (1/2)m(B_1)$. Similarly, for a set of $\delta$-tubes $\T$, we say these tubes are essentially distinct if for any $T_1\neq T_2\in\T$, $m(T_1\cap T_2)\leq (1/2)m(T_1)$. Here $m(X)$ stands for the Lebesgue measure of set $X$.
\end{definition}

In the rest of the paper, we will always consider essentially distinct $\delta$-balls and essentially distinct $\delta$-tubes.

In the discrete case, it's easy to define the incidence between points and lines, and to define the $r$-rich points and $r$-rich lines. Here we make analogies of these notions for $\de$-balls and $\de$-tubes.

\begin{definition}[$r$-rich balls and $r$-rich tubes]\label{rich}
Given a set of $\delta$-tubes $\mathbb{T}$, we define the $r$-rich balls for $\T$ in the following way. We choose a set $\B$ to be a maximal set of essentially distinct $\delta$-balls.
We define
$$B_r(\mathbb{T}):=\{ B\in\B: B\ \text{intersects more than}\ r\  \text{tubes from}\ \T \}.$$
We say $B_r(\T)$ is the set of $r$-rich $\delta$-balls for $\T$. 
Here we have many choices for $\B$, but we will see in the proof that the choice of $\B$ doesn't affect the result for the upper bound of $|B_r(\T)|$.
We could just choose $\B$ to be all $\delta$-balls centered at $(\delta/2) \mathbb{Z}^2$).

Similarly, given a set of $\delta$-balls $\B$, we define the $r$-rich tubes for $\B$ in the following way. We choose a set $\T$ to be a maximal set of essentially distinct $\delta$-tubes. We define
$$T_r(\B):=\{ T\in\T: T\ \text{intersects more than}\ r\ \text{balls from}\ \B \}.$$ 
We say $T_r(\B)$ is the set of $r$-rich $\delta$-tubes for $\B$. 
\end{definition}


Now we state our main results.

\begin{theorem}\label{main2}
Let $1 \le W \le X \le \delta^{-1}$. Let $\T$ be a collection of essentially distinct $\delta$-tubes in $[0,1]^2$. We also assume $\T$ satisfies the following spacing condition: every $W^{-1}$-tube contains at most $\frac{X}{W}$ many tubes of $\T$, and the directions of these tubes are $\frac{1}{X}$-separated. 

We denote $|\T_{\max}|:=WX$ (as one can see that $\T$ contains at most $\sim WX$ tubes).
Then for $r > \max(\delta^{1-2\e} |\Tmax|, 1)$, the number of $r$-rich balls is bounded by
\begin{equation}\label{eq1}
    |B_r (\T)| \lesim_\e \delta^{-\e} |\T| |\Tmax| \cdot r^{-2} (r^{-1} + W^{-1}).
\end{equation}
\end{theorem}

\begin{remark}
If we take $X=W$ in the above theorem and note that in this case  $|B_r(\T)|=0$ for $r>W$, we recover Theorem 1.1 in \cite{guth2019incidence}.
If we take $X=\delta^{-1}$ in the above theorem and note that in this case  $r>\de WX=W$, we recover Theorem 1.2 (with $N_1=1$) in \cite{guth2019incidence}.
There are two new ingredients in our theorem. First, we use $|\T||\T_{\max}|$ as our upper bound in \eqref{eq1}, whereas in \cite{guth2019incidence} it was $|\T_{\max}|^2$. Second,  
our theorem also concerns about the intermediate spacing conditions, i.e. we introduce a new parameter $X$.
\end{remark}

There is another version of Theorem \ref{main2}. To motivate our idea, we introduce two notions: direction and position. Any $\de$-tube contained in $[0,1]^2$ that forms an angle $\ge\frac{\pi}{4}$ with the $x$-axis is determined by its direction, as well as its intersection with $x$-axis (which we call its position). Switching the role of direction and position gives us the following theorem.

\begin{theorem}\label{main3}
Fix a line $\ell$ that intersects $[0, 1]^2$. Let $1 \le W \le X \le \delta^{-1}$. Let $\T$ be a collection of essentially distinct $\delta$-tubes in $[0,1]^2$, such that every tube in $\T$ forms an angle $\ge \frac{\pi}{4}$ with $\ell$. We also assume $\T$ satisfies the following spacing condition: every $W^{-1}$-tube which form an angle $\ge \frac{\pi}{4}$ with $\ell$ contains at most $\frac{X}{W}$ many tubes of $\T$, and the intersections of these tubes with $\ell$ are $\frac{1}{X}$-separated.

We denote $|\T_{\max}|:=WX$ (as one can see that $\T$ contains at most $\sim WX$ tubes).
Then for $r > \max(\delta^{1-2\e} |\Tmax|, 1)$, the number of $r$-rich balls is bounded by 
\begin{equation*}
    |B_r (\T)| \lesim_\e \delta^{-\e} |\T| |\Tmax| \cdot r^{-2} (r^{-1} + W^{-1}).
\end{equation*}
\end{theorem}

\begin{remark}
The above two theorems give upper bounds for the number of $r$-rich balls $B_r(\T)$ when tubes $\T$ satisfy some spacing condition. We can actually switch the roles of balls and tubes, so the question becomes to estimate the number of $r$-rich tubes $T_r(\B)$ assuming some spacing condition on $\B$. This is our Theorem \ref{main} stated below. In Section \ref{dualitysec}, we will discuss a tube-ball duality and show that Theorem \ref{main} implies Theorem \ref{main2} and Theorem \ref{main3}. 
\end{remark}

\begin{theorem}\label{main}
Let $1 \le W \le X \le \delta^{-1}$. Divide $[0, 1]^2$ into $W^{-1} \times X^{-1}$ rectangles as in Figure \ref{fig:generalcase}. Let $\B$ be a set of $\delta$-balls with at most one ball in each rectangle. 

We denote $|\B_{\max}|:=WX$ (as one can see that $\B$ contains at most $\sim WX$ balls).
Then for $r > \max(\delta^{1-2\e} |\B_{\max}|, 1)$, the number of $r$-rich tubes is bounded by
\begin{equation*}
    |T_r (\mathbb{B})| \lesim_\e \delta^{-\e} |\B| |\B_{\max}| \cdot r^{-2} (r^{-1} + W^{-1}).
\end{equation*}
\end{theorem}

\begin{remark}
There are two special cases of Theorem \ref{main}: $X=W$, $X=\de^{-1}$ (see Figure \ref{fig:specialcases}). These two cases actually correspond to (the dual version of) Theorem 1.1 and Theorem 1.2 in \cite{guth2019incidence}.
\end{remark}

\begin{figure}
    \centering
    \includegraphics{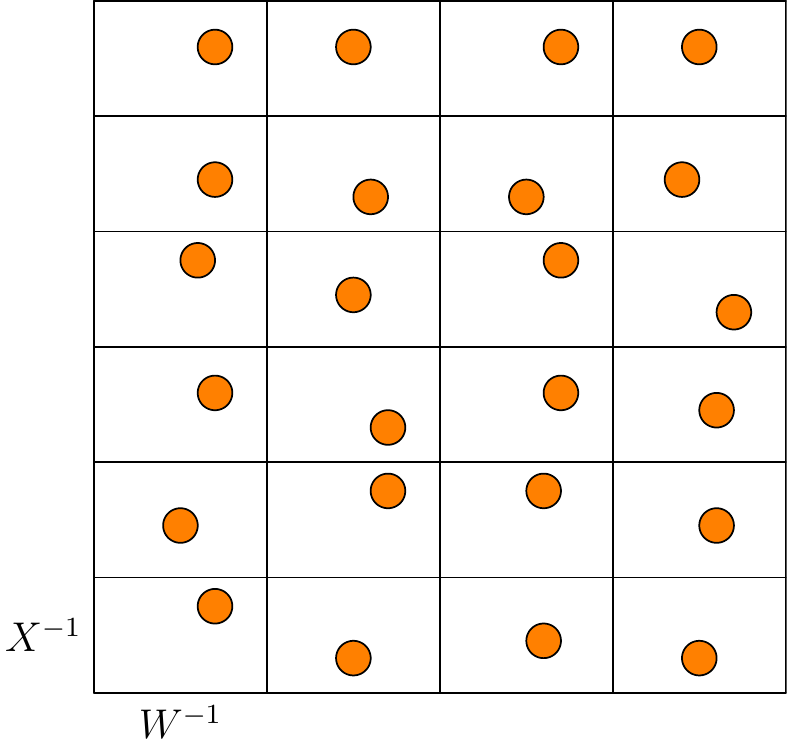}
    \caption{The general case of Theorem \ref{main}.}
    \label{fig:generalcase}
\end{figure}

\begin{figure}
    \centering
\includegraphics{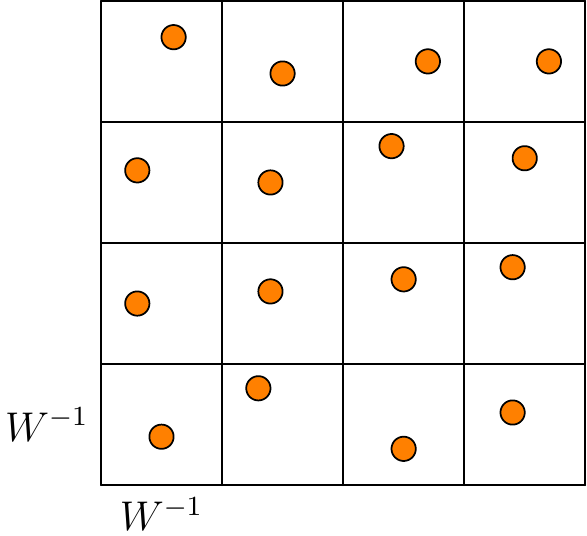}
\includegraphics{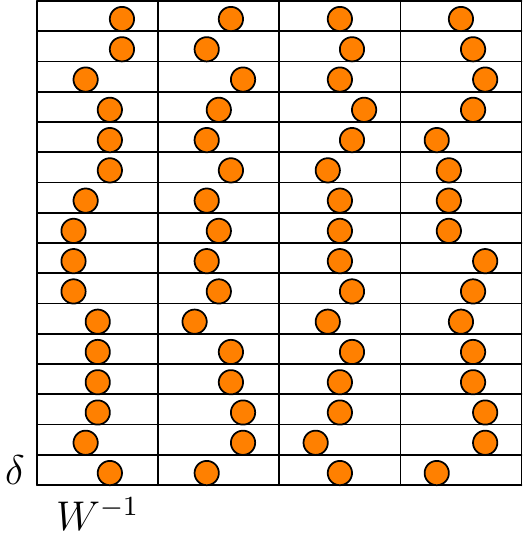}
    \caption{Special cases of Theorem \ref{main}.}
    \label{fig:specialcases}
\end{figure}

\subsection{Furstenberg set problem}
Wolff discussed the Furstenberg set problem in \cite{wolff1999recent}. Given $\alpha\in (0,1)$, we say a set $E\subset \R^2$ is an $\alpha$-Furstenberg set if for each direction $e\in S^1$, there exits a line $l_e$ pointing in direction $e$ such that $\dim_{\textup{H}}(E\cap l_e)\ge \alpha$. The problem is to find the lower bound of $\dim_{\textup{H}}E$. Wolff proved that $\dim_{\textup{H}}E\ge \max(\frac{1}{2}+\alpha, 2\alpha)$ and conjectured that $\dim_{\textup{H}}E\ge \frac{3}{2}\alpha+\frac{1}{2}$.

Some progress has been made on this problem. In \cite{katz2001some}, Katz and Tao showed that when $\alpha=\frac{1}{2}$, the Furstenberg problem is related to other two problems: Falconer distance problem and Erd\"os ring problem. Later, Bourgain \cite{bourgain2003erdHos} improved the bound to $\dim_{\textup{H}}E\ge 1+\e$ when $\alpha=\frac{1}{2}$. Recently, Orponen and Shmerkin \cite{orponen2021hausdorff} further improved the bound to $\dim_{\textup{H}}E\ge 2\alpha+\e$ for $\alpha\in(\frac{1}{2},1)$.

A general Furstenberg set problem was also considered by many authors, for example in \cite{molter2012furstenberg}, \cite{hera2020improved}, \cite{orponen2021hausdorff}. For $(u_0,v_0)\in[0,1]^2$, define the line $l(u_0,v_0): v_0y=x-u_0$.  We say $E\subset \R^2$ is a $(\alpha,\beta)$-Furstenberg set, if there exists an $\beta$-dimensional set $X\in [0,1]^2$ such that for each line $l\in\{ l(u,v):(u,v)\in X \}$, we have $\dim_{\text{H}}(l\cap E)\ge \alpha$. The problem is to find the lower bound of $\dim_{\textup{H}}E$.

There are also some variants of the Furstenberg problem. In \cite{zhang2017polynomials}, Zhang considered the discrete Furstenberg problem and proved the sharp estimates.

In our paper we consider the $\de$-discretized version assuming an evenly spacing condition. We consider the following question.

\begin{question}\label{question}
Fix $\alpha\in (0,1)$. Let $\T=\{T\}$ be a set of $\de$-tubes that are $\de$-separated in direction, and with cardinality $\sim \de^{-1}$. Assume for each $T$ there is a set of $\de$-balls $Y(T)=\{B_\de\}$ satisfying: each ball in $Y(T)$ intersects $T$; $\#Y(T)\sim \de^{-\alpha}$ and the balls in $Y(T)$ have spacing $\gtrsim \de^{\alpha}$. 

If we define the union of these $\de$-balls to be $\B=\cup_T Y(T)$, can we show
$$ |\B| \gtrapprox \de^{-\frac{3}{2}\alpha-\frac{1}{2}}? $$
\end{question}

We will give an affirmative answer to this question in Section \ref{fursec}. Actually, we will prove a more general result as follows, which could be thought of as the evenly spaced $(\alpha,\beta)$-Furstenberg problem.

\begin{theorem}[Evenly spaced Furstenberg]\label{thmfur}
Let $1 \le W \le X \le \delta^{-1}$. Let $\T$ be a collection of essentially distinct $\delta$-tubes in $[0,1]^2$ that satisfies the following spacing condition: every $W^{-1}$-tube contains at most $\frac{X}{W}$ many tubes of $\T$, and the directions of these tubes are $\frac{1}{X}$-separated. We also assume $|\T|\sim XW$.

Assume for each $T$ there is a set of $\de$-balls $Y(T)=\{B_\de\}$ satisfying: each ball in $Y(T)$ intersects $T$; $\#Y(T)\sim \de^{-\alpha}$ and the balls in $Y(T)$ have spacing $\gtrsim \de^{\alpha}$. 
Define the union of these $\de$-balls to be $\B=\cup_T Y(T)$.
Then we have the estimate
\begin{equation}\label{givenest}
    |\B|\gtrsim (\log\de^{-1})^{3.5}\min(\de^{-\alpha-1},\de^{-\frac{3}{2}\alpha}(XW)^{\frac{1}{2}},\de^{-\alpha}XW).
\end{equation}
\end{theorem}

Question \ref{question} is a special case of Theorem \ref{thmfur} when $W=1, X=\de^{-1}$.

\begin{remark}
In our theorem, the $Y(T)$ satisfies an evenly spacing condition which is stronger than the $(\de,\alpha)_1$ spacing condition introduced in \cite{katz2001some}. The $(\de,\alpha)_1$ spacing condition roughly says that $\#Y(T)\sim\de^{-\alpha}$ and for any $\de\times w$-subtube $T_w \subset T\cap [0,1]^2$ there holds $\#\{B_\de\in Y(T): B_\de\cap T_w\neq\emptyset\}\lesim (w/\de)^{\alpha}$.
Our tube set $\T$ also satisfies an evenly spacing condition which is stronger than the $(\de,\beta)_2$ spacing condition: $|\T|\sim \de^{-\beta}$ (one may think $\de^{-\beta}=XW$); any $w\times 1$ tube in $[0,1]^2$ contains $\lesim (w/\de)^\beta$ many tubes in $\T$. It was shown in \cite{hera2020improved} Lemma 3.3 that the $\de$-discretized version under the $(\de,\beta)_2$-condition for $\T$ and $(\de,\alpha)_1$-condition for $Y(T)$ will imply the original Furstenberg problem (in terms of Hausdorff dimension).
\end{remark}

\begin{remark}
Given the sharp estimate $\eqref{givenest}$ under the evenly spacing condition, it might be reasonable to ask whether for any $(\alpha,\beta)$-Fustenberg set $E$ we have
$$ \dim_{\text{H}}E\ge \min(\alpha+1,\frac{3}{2}\alpha+\frac{1}{2}\beta,\alpha+\beta)? $$
\end{remark}

\medskip

To end this section, we discuss the plan of this paper. In Section \ref{egsec}, we discuss the sharp examples. In Section \ref{dualitysec}, we discuss the tube-ball duality and show Theorem \ref{main} implies Theorem \ref{main2} and Theorem \ref{main3}. In Section \ref{proofsec}, we prove Theorem \ref{main}. In Section \ref{fursec}, we prove Theorem \ref{thmfur}.

\medskip

\noindent
{\bf Notation.} We use the notation $A\lesim B$ to mean $A\le CB$ for some constant $C>0$.
We use the notation $A\lessapprox B$ in several sections. The meaning of this notation may be slightly different in different places, but the precise definition is given where it appears.

\medskip

\noindent
{\bf Acknowledgements.} We would like to thank Prof. Larry Guth for helpful discussions. We would also like to thank the referee for a careful reading and many helpful suggestions.

\section {Sharp examples} \label{egsec}
In this subsection, we discuss the sharp example for Theorem \ref{main2} when $|\T|\sim |\T_{\max}|$ and Theorem \ref{thmfur}.

The sharp examples for Theorem \ref{main3} and Theorem \ref{main} can be constructed in a similar way as for Theorem \ref{main2} by using the tube-ball duality (which will be discussed in Section \ref{dualitysec}). So we omit the construction for other two theorems.

\subsection{Examples for incidence estimate}\label{egincidence}
First, we construct the example for Theorem \ref{main2}.
For simplicity, we assume $W \mid X$ ($W$ divides $X$).

\textit {Case 1}: $2\le r < W$. 

For each $0 \le a \le W$ and $0 \le b \le X$, draw a line from $(\frac{a}{W},0)$ to $(\frac{b}{X},1)$. These lines, when thickened to $\delta$-tubes, will satisfy the spacing condition as in Theorem \ref{main2}, since two lines are either parallel or differ by angle $\frac{1}{X} \ge \delta$. Let $S$ be the set of rational numbers $\frac{p}{q}$ in $[\frac{1}{4}, \frac{3}{4}]$ such that $\frac{X}{100r}\le p, q \le \frac{100X}{r}$, $\gcd(p,q)=1$, and $p$ is a multiple of $\frac{X}{W}$. We claim that each point of the form $( \frac{c}{q W}, \frac{p}{q})$ with $\frac{p}{q} \in S$ and $c \le qW$ is $r$-rich. To see this, note that the point on the line through $(\frac{a}{W},0)$ and $(\frac{b}{X},1)$ with $y$-coordinate $\frac{p}{q}$ has $x$-coordinate $\frac{p}{q} \cdot \frac{b}{X} + (1-\frac{p}{q}) \cdot \frac{a}{W}$, so it suffices to show the equation
$$\frac{p}{q} \cdot \frac{b}{X} + (1-\frac{p}{q}) \cdot \frac{a}{W}=\frac{c}{qW}$$
has $\gtrsim r$ solutions $(a,b)\ (0\le a\le W,0\le b\le X)$, for any $\frac{p}{q}\in S$ and $c\le qW$. Multiplying by $qW$, the equation is equivalent to
\begin{equation}\label{exeq}
    \frac{pW}{X}\cdot b + (q-p) \cdot a=c.
\end{equation}

Note that $\frac{pW}{X}$ is an integer since we assumed $p$ is a multiple of $\frac{X}{W}$.
We also have $\gcd(\frac{pW}{X}, q-p)=1$, since $\gcd(p, q-p) = 1$. Now we can show \eqref{exeq} has $\gtrsim r$ solutions $(a, b)$. Note that if $(a_0, b_0)$ is a solution (there is always a solution since $\gcd(\frac{pW}{X}, q-p)=1$), then $(a_0+\frac{pW}{X}m, b_0-(q-p)m), m\in\mathbb{Z}$ are also solutions. When $c\le qW$, we can properly choose $\gtrsim r$ many $m\in\mathbb{Z}$ such that $a_0+\frac{pW}{X}m\in[0,W]$ and $b_0-(q-p)m\in [0,X]$.

We still need to check the points $\{ (\frac{c}{qW},\frac{p}{q}): \frac{p}{q}\in S, c\le qW \}$ are $\de$-separated. Consider two different points $(\frac{p}{q} \cdot \frac{b}{X} + (1-\frac{p}{q}) \cdot \frac{a}{W},\frac{p}{q})$ and $(\frac{p'}{q'} \cdot \frac{b'}{X} + (1-\frac{p'}{q'}) \cdot \frac{a'}{W},\frac{p'}{q'})$ in this set. If their second coordinates are different, then since we assumed each of $p$, $p'$ is a multiple of  $\frac{X}{W}$, we see the difference of their second coordinates is
\begin{equation}\label{difference}
   \big|\frac{p}{q}-\frac{p'}{q'}\big|=\big| \frac{p q'-p' q}{q q'}\big| \ge  \frac{X/W}{(X/r)^2}\ge \frac{r}{XW}\ge \de,
\end{equation}
where the last inequality is because of the assumption $r>\de |\T_{\max}|=\de WX$ in Theorem \ref{main2}.
If their second coordinates are same, then since their first coordinates are of form $\frac{c}{qW}$ and $ \frac{c'}{qW}$, we see the difference of their first coordinates is
$$\ge \frac{1}{qW}\sim \frac{r}{XW}\ge \de. $$

Finally, we calculate the cardinality of the set of $r$-rich points: $\{ (\frac{c}{qW},\frac{p}{q}): \frac{p}{q}\in S, c\le qW \}$.
There are $\frac{WX}{r^2}$ elements in $S$ and $\sim \frac{WX}{r}$ choices for $c$, so the number of $r$-rich points is $\sim\frac{W^2 X^2}{r^3}$.

\medskip
\textit{Case 2}: $W < r < X$. 

At each $(\frac{a}{W}, 0)$ with $0 \le a \le W$, place an $X$-bush, i.e. a set of $X^{-1}$-direction separated $\de$-tubes with cardinality $X$. The number of $r$-rich $\de$-balls in each bush is $\sim\frac{X^2}{r^2}$. Actually these $r$-rich points are contained in the ball centered at $(\frac{a}{W},0)$ of radius $\frac{X}{r}\de$. Also note that the the $r$-rich balls from different bush are disjoint (since $\frac{X}{r}\de\le W^{-1}$),  so the total number of $r$-rich points is $\sim\frac{WX^2}{r^2}$.\qed

\subsection{Examples for Furstenberg problem}
Next we discuss the sharp examples for Theorem \ref{thmfur}. Without loss of generality, we may assume the directions of tubes in $\T$ are within $1/10$ angle with the $y$-axis. We also assume $X$ and $W$ are square numbers and $W \mid X$ for technical reasons.

\textit{Case 1}:
$\de^{-\alpha-1}$. 

Choose $\sim \de^{-\alpha}$ many length-$\de$ intervals in $[0,1]$ such that any two intervals have distance $\ge \de^{\alpha}$ from each other. Denote these intervals by $\{I_i\}_i$. We set $\B$ to be all the lattice $\de$-balls that intersect $[0,1]\times\cup_i I_i $. We can easily check $\B$ satisfies the condition in Theorem \ref{thmfur} for any choice of tubes, and 
$$ |\B|\lesim \de^{-\alpha-1}. $$

\medskip
\textit{Case 2}: $\de^{-\alpha}XW$. 

First we fix a set of tubes $\T$ that satisfies the spacing condition. Let $\B$ be the set of lattice $\de$-balls that intersect $([0,1]\times \cup_i I_i) \bigcap \cup_{T\in \T} T $, where $I_i$ are the same as in \textit{Case 1}. Noting $|\T|\sim XW$, we can easily check that
$$ |\B|\lesim \de^{-\alpha}XW. $$

\medskip
\textit{Case 3}: $\de^{-\frac{3}{2}\alpha}(XW)^{\frac{1}{2}}$.

We will borrow the idea from \textit{Case 1} of the examples for incidence estimate in the last subsection. The notation here will be the same as there. We choose the same set of tubes $\T$ as in \textit{Case 1} of the last subsection. We choose $\B$ to be the set of $\de$-balls whose centers are from the set
$\{ (\frac{c}{qW},\frac{p}{q}): \frac{p}{q}\in S, c\le qW \}$. We have
$$ |\B|\lesim \frac{W^2X^2}{r^3}. $$

We check that $\B$ satisfies the condition in Theorem \ref{thmfur}. Fix a $T\in \T$, let the line connecting $(\frac{a}{W},0)$ and $(\frac{b}{X},1)$ be the core line of $T$. We see that the points $\{((1-\frac{p}{q})\frac{a}{W}+\frac{p}{q}\frac{b}{X},\frac{p}{q}):\frac{p}{q}\in S\}$ lie on the core line of $T$. We can also show that these points belong to the set $\{ (\frac{c}{qW},\frac{p}{q}): \frac{p}{q}\in S, c\le qW \}$, since
$$ (1-\frac{p}{q})\frac{a}{W}+\frac{p}{q}\frac{b}{X}=\frac{qW-pa+p\frac{W}{X}b}{qW} $$
and noting that $p$ is a multiple of $\frac{X}{W}$.

We have shown that for each number $\frac{p}{q}\in S$, the core line of $T$ contains points whose $y$-axis is $\frac{p}{q}$. Recall that $\#S\sim \frac{XW}{r^2}$, and from \eqref{difference} we see that
each pair of nearby points have distance $\ge \frac{r^2}{XW}$. If we have $2\le (\de^\alpha XW)^{\frac{1}{2}}\le W $,
then we set $r=(\de^\alpha XW)^{\frac{1}{2}}$ which means $\de^{-\alpha}=\frac{XW}{r^2}$. A simple calculation gives
$$ |\B| \lesim \frac{W^2X^2}{r^3}=\de^{-\frac{3}{2}\alpha}(XW)^{\frac{1}{2}}. $$
Next, we will get rid of the requirement $2\le (\de^{\alpha}XW)^{\frac{1}{2}}\le W$. 

When $(\de^\alpha XW)^{\frac{1}{2}}\le 2$, we just use the example in \textit{Case 2} and note that $\de^{-\alpha}XW\lesim \de^{-\frac{3}{2}\alpha}(XW)^{\frac{1}{2}}$.

Now we can assume $(\de^\alpha X W)^{\frac{1}{2}}\ge 2$. We set another pair $(X',W')=((XW)^{\frac{1}{2}},(XW)^{\frac{1}{2}})$. 
We can easily check $2\le(\de^{\alpha}X'W')^{\frac{1}{2}}\le W'$. If we just construct the sets $\T, \B$ as above using parameters $(W',X')$, we have
$$ |\B| \lesim \de^{-\frac{3}{2}\alpha}(X'W')^{\frac{1}{2}}=\de^{-\frac{3}{2}\alpha}(XW)^{\frac{1}{2}}. $$
However, our problem is: with the new pair $((XW)^{\frac{1}{2}},(XW)^{\frac{1}{2}})$, the $\T$ doesn't satisfy the spacing condition in Theorem \ref{thmfur}. We overcome this by using the irrational translation trick in \cite{wolff1999recent}. We slightly modify the definition $\T$. For each $0\le a,b\le (XW)^{\frac{1}{2}}$, draw a line segment from $(\frac{a}{(XW)^{\frac{1}{2}}},0)$ to $(\frac{\sqrt{2}b}{(XW)^{\frac{1}{2}}},0)$. We define $\T$ to be the set of tubes that are the $\de$-neighborhoods of these line segments. The intersection pattern of tubes and balls are the same with $\sqrt{2}$ replaced by $1$, so we still get the bound
$$ |\B| \lesim \de^{-\frac{3}{2}\alpha}(XW)^{\frac{1}{2}}. $$
But now, $\T$ satisfies the spacing condition in Theorem \ref{thmfur}. To check this, we consider any $W^{-1}$-tube whose intersection with $\{y=0\}$ is $[a_0,a_0+W^{-1}]$ and intersection with $\{y=1\}$ is $[b_0,b_0+W^{-1}]$. We see that the line segments that lie in this $W^{-1}$-tube are those connecting points $(a_0+\frac{a}{(XW)^{\frac{1}{2}}},0)$ and $(b_0+\frac{\sqrt{2}b}{(XW)^{\frac{1}{2}}},1)$ for $0\le a,b \lesssim (\frac{X}{W})^{\frac{1}{2}}$. 
It suffices to show $\frac{a-\sqrt{2}b}{(XW)^{\frac{1}{2}}}\gtrsim \frac{1}{X}$, which is a simple result of the fact that $|a-\sqrt{2}b|= \frac{|a^2-2b^2|}{|a+\sqrt{2}b|}\ge\frac{1}{|a+\sqrt{2}b|}\ge \frac{1}{4\max(a,b)} $.
\qed

\section{Tube-ball duality} \label{dualitysec}
We know there is a duality between lines and points. More precisely, in the projective plane every point has its dual line and every line has its dual point. A point and a line intersect if and only if their dual line and dual point intersect. So, we can transform the point-line incidence into line-point incidence. 

In this section, we are going to show there is also a duality between $\de$-tubes and $\de$-balls that lie in (or near) $[0,1]^2$. The advantage is that we can transform Theorem \ref{main2} and Theorem \ref{main3} into Theorem \ref{main}. We assume all the $\de$-tubes and $\de$-balls considered here lie in $\Pi_1=[0,1]^2$, which we call the \textit{physical space}. We also set $\Pi_2=[0,1]^2$ which we call the \textit{dual space}. Our goal is to define a correspondence between these two spaces so that: the $\de$-balls (respectively $\de$-tubes) in $\Pi_1$ correspond to $\de$-tubes (respectively $\de$-balls) in $\Pi_2$, and the ball-tube incidence in $\Pi_1$ correspond to the tube-ball incidence in $\Pi_2$. A similar discussion for such duality can be found in \cite{orponen2021hausdorff} (Section 2.3).


\subsection{Line-point duality} 
First, let's look at the line-point duality between $\Pi_1$ and $\Pi_2$. We will use $(x,y)$ to denote the coordinates of $\Pi_1$ and $(u,v)$ to denote the coordinates of $\Pi_2$.

Define $\cP_2$ to be all the points in $\Pi_2$. For any $(u_0,v_0)\in\cP_2$, we define the corresponding line in $\Pi_1$ to be $$l_1(u_0,v_0): v_0 y=x-u_0.$$ 
We also define 
$$\cL_1:=l_1(\cP_2)=\{ v_0 y=x-u_0:(u_0,v_0)\in\cP_2 \},$$
which is a set of lines in $\Pi_1$.
We see $l_1: \cP_2\leftrightarrow\cL_1$ is a one-to-one correspondence.

\begin{remark}\label{rmk}
There is a good way to think about this correspondence. Given a point $(u_0,v_0)$, then its corresponding line $l_1$ has ``position" $u_0$ (which is its intersection with $\{y=0\}$) and has ``direction" $v_0$ (which is the inverse of its slope). In the next subsection, we will define a correspondence between balls in $\Pi_2$ and tubes in $\Pi_1$ so that a ball with center $(u_0,v_0)$ corresponds to the tube with ``position" $u_0$ and ``direction" $v_0$.
\end{remark}

Next, we define $\cP_1$ to be all the points in $\Pi_1$. For any $(x_0,y_0)\in \cP_1$, we know the lines passing through it are of the form $v y=x-x_0+v y_0$. This motivates us to define the line in $\Pi_2$ corresponding to $(x_0,y_0)$ as 
$$l_2(x_0,y_0): u=x_0-v y_0.$$
We also define 
$$\cL_2=l_2(\cP_1)=\{ u=x_0-v y_0: (x_0,y_0)\in\cP_1 \},$$
which is a set of lines in $\Pi_2$.
We see $l_2:\cP_1\leftrightarrow\cL_2$ is a one-to-one correspondence.

We can also show the incidence is preserved under the duality. Given a point $(x_0,y_0)\in\cP_1$ and a line $l_1(u_0,v_0):v_0 y=x-u_0 \in\cL_1$, we have $(x_0,y_0)\in l_1(u_0,v_0)\Longleftrightarrow (u_0,v_0)\in l_2(x_0,y_0)$ by definition.

\subsection{Tube-ball duality}\label{tubeball}
Now we generalize our line-point duality to tube-ball duality.

For $(u_0,v_0)\in \cP_2$, let $B=B_\de(u_0,v_0)$ be the ball of radius $\de$ with center $(u_0,v_0)$. The intersection of its image under $l_1$ with $[-2,2]^2$ is roughly a $\de$-tube. That is to say:
$$ l_1(B):=\bigcup_{(u,v)\in B}l_1(u,v)\bigcap [-2,2]^2 $$
is roughly a $\de$-tube.
Intuitively, one can think of $l_1(B)$ as the $\de$-neighborhood of $l_1(u_0,v_0)\bigcap [-2,2]^2$. If we let $\B_2$ be all the lattice $\de$-balls in $\Pi_2$, and let $\T_1:=\{ l_1(B):B\in\B_2 \}$, then $l_1: \B_2\leftrightarrow \T_1$ is a one-to-one correspondence.
We can similarly define $l_2, \B_1$ and $\T_2$, so that $l_2:\B_1 \leftrightarrow \T_2$ is a on-to-one correspondence.

Moreover, we can check the incidence is preserved under the duality, i.e. given a ball $B_1\in\B_1$ and a tube $T_1=l_1(B_2)\in\T_1$, then $(B_1,T_1)$ counts one incidence in $\Pi_1$ if and only if $(l_1^{-1}(T_1),l_2(B_1))=(B_2,T_2)$ counts one incidence in $\Pi_2$.

To get a better understanding of this tube-ball duality, see Figure \ref{dual}. Here, for each orange ball $B$ in $\Pi_2$, there is a corresponding orange tube $l_1(B)$ in $\Pi_1$. Similarly, for each blue ball $B'$ in $\Pi_1$, there is a corresponding blue tube $l_2(B')$ in $\Pi_2$. Also the incidence is preserved in the sense that the orange tube and the blue ball intersect if and only if the corresponding orange ball and blue tube intersect.

\begin{figure}
    \centering
    \includegraphics{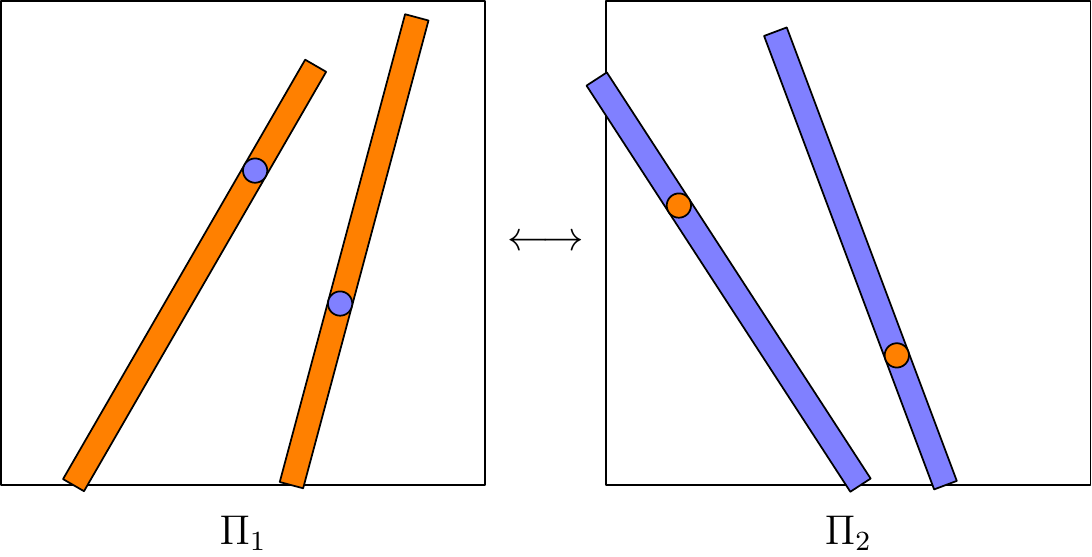}
    \caption{Tube-ball duality.}
    \label{dual}
\end{figure}


\subsection{Relations between the theorems}\label{subsec3.3}
We prove that Theorem \ref{main} implies Theorem \ref{main2} and Theorem \ref{main3} in this subsection.

As mentioned in the beginning of this section, we can use this duality to turn from ball-tube incidence to tube-ball incidence.
For example, if we are given a set of $\de$-balls $\B$ and $\de$-tubes $\T$ and $\T$ satisfying some spacing condition, then by duality this is equivalent to the problem for a set of $\de$-balls $\B'$ and $\de$-tubes $\T'$ with $\B'$ satisfying a similar spacing condition. What we did is we transfer the spacing condition from tubes to balls. This gives the heuristic that Theorem \ref{main2} (or Theorem \ref{main3}) can be reduced to Theorem \ref{main}.

However, there is still a shortage that the tubes $\T_i\  (i=1,2)$ we defined do not contain all the tubes lying in $\Pi_i\ (i=1,2)$. For example, all the tubes in $\T_1$ have slopes in $[-\infty,-1]\cup[1,\infty]$, which means $\T_1$ only contains the tubes that form an angle $\le \frac{\pi}{4}$ with $y$-axis. However, we can find several rotations $\{\rho_k\}_{k\le 100}$ (for example, $\rho_k$ is the rotation with angle $\frac{2\pi k}{100}$ and center $(\frac{1}{2},\frac{1}{2})$), so that $\bigcup_{k=1}^{100} \rho_k(\T_1)$  are morally all the $\de$-tubes in $\Pi_1$. Since $\B_1$ is all the $\de$-balls in $\Pi_1$, $\B_1$ is morally the same under any rotation: $\rho_k(\B_1)=\B_1$.

Let us see how this work. Suppose we are given a set of $\de$-tubes $\T$ which satisfies some spacing condition. We want to estimate the number of $r$-rich balls $B_r(\T)$. By pigeonholing, we have
$$ |B_r(\T)|\lesim \sum_{k=1}^{100} |B_{r/100}(\T\cap \rho_k(\T_1))|=\sum_{k=1}^{100} |B_{r/100}(\rho_k^{-1}(\T)\cap \T_1)|.$$
By the tube-ball duality, it is bounded by
$$\sum_{k=1}^{100} |T_{r/100}(\B_k')|.$$
where $\B'_k=l_1^{-1}(\rho_k^{-1}(\T)\cap \T_1)$. Now $\B_k'$ inherits the same spacing condition from $\T$, so it suffices to estimate the number of $r/100$-rich tubes assuming $\de$-balls satisfy some spacing condition.

To prove that Theorem \ref{main} implies Theorem \ref{main2} (or Theorem \ref{main3}), we only need to verify: If $\T$ is a set of tubes in $\T_1$ that satisfies spacing condition in Theorem \ref{main2} (Theorem \ref{main3}), then the set of $\de$-balls $\B=\{l_1^{-1}(T):T\in \T\}$ satisfies the spacing condition in Theorem \ref{main}.

First, we suppose $\T$ is a set of tubes in $\T_1$ that satisfies spacing condition in Theorem \ref{main2}. That is, any $W^{-1}$-tube contains at most $\frac{X}{W}$ many tubes of $\T$, and the directions of these tubes are $\frac{1}{X}$-separated. For any $W^{-1}$-ball $B_{W^{-1}}$ in $\Pi_2$ with center $(u_0,v_0)$, consider the $W^{-1}$-tube $T_{W^{-1}}$ with ``position" $u_0$ and ``direction" $v_0$ (see Remark in section \ref{rmk}), i.e. $T_{W^{-1}}$ is the $W^{-1}$-neighborhood of $v_0y=x-u_0$. We see that the map $l_1$ induce a correspondence between the $\de$-balls lying in $B_{W^{-1}}$ and the $\de$-tubes lying in $T_{W^{-1}}$. By the spacing condition, the tubes $T\in\T$ that lie in $T_{W^{-1}}$ are $\frac{1}{X}$-separated in direction, so the corresponding balls in $B_{W^{-1}}$ have $\frac{1}{X}$-separated $v$-coordinates. That means, if we partition $B_{W^{-1}}$ into about $\frac{X}{W}$ many $W^{-1}\times X^{-1}$-rectangles (the long side of the rectangles point to the direction of $u$-axis), we have that in each rectangle there is at most one $\de$-ball from $\B=\{l^{-1}_1(T): T\in\T\}$. Since our $B_{W^{-1}}$ can be any $W^{-1}$-ball, we see that $\B$ satisfies the spacing condition in Theorem \ref{main}.

Similarly we could make the same argument as above for Theorem \ref{main3} by switching the role of $u$-coordinate and $v$-coordinate in $\Pi_2$.
First, we may assume the line $\ell$ in Theorem \ref{main3} is parallel to $x$-axis by rotation. Next, we may assume $\ell$ is $\{y=0\}$, otherwise we just consider the incidence estimate in the portion of $[0,1]^2$ above $\ell$ and the portion of $[0,1]^2$ below $\ell$ separately. If the $\ell$ in Theorem \ref{main3} is $\{y=0\}$, we can prove the following result:
Let $\T$ be a set of tubes in $\T_1$ that satisfies the spacing condition in Theorem \ref{main3}. Partition $\Pi_2=[0,1]^2$ into $X^{-1}\times W^{-1}$-rectangles (the long side of the rectangles now point to the direction of $v$-axis which is different from that in the last paragraph). Then, each rectangle contains at most one ball from $\B=\{l^{-1}_1(T): T\in\T\}$. Since the proof is similar, we omit the proof.

\section{Proof of Theorem \ref{main}}\label{proofsec}
In this section, we prove Theorem \ref{main}.
We will first prove two lemmas and then use them to finish the proof of Theorem \ref{main}.

\subsection{Two lemmas}
First, we will need the ``dual version'' of Proposition 2.1 from \cite{guth2019incidence}, which was inspired by ideas of Orponen \cite{orponen2018dimension} and Vinh \cite{vinh2011szemeredi}. We state the version for $n = 2$. The dual version just follows from the original one (Proposition 2.1 in \cite{guth2019incidence}) by the tube-ball duality discussed in Section \ref{dualitysec}, so we omit the proof.

\begin{prop}\label{twopointone}
Fix a tiny $\eps > 0$. There exists a constant $C(\eps)$ with the following property: Suppose that $\B$ is a set of unit balls in $[0, D]^2$ and $\T$ is a set of essentially distinct tubes of length $D$ and radius $1$ in $[0, D]^2$. Suppose that each tube of $\T$ contains about $E$ balls of $\B$. Let $S = D^{\eps/20}$. Then either:

\textbf{Thin case.} $|\T| \leapp S^2 E^{-2} |\B| D$, or

\textbf{Thick case.} There is a set of finitely overlapping $2S \times D$-tubes $U_j$ (heavy tubes) such that:
\begin{enumerate}[(1)]
    \item $\bigcup_j U_j$ contains a fraction $\gtrapprox 1$ of the tubes of $\T$;
    
    \item Each $U_j$ contains $\geapp SE$ balls of $\B$.
\end{enumerate}
In particular, if we define $\wt\T$ to be the set of $\gtrapprox SE$-rich $2S\times D$-tubes, we have
\begin{equation}\label{tube}
    |\T|\lessapprox S^2 (E^{-2}|\B|D+|\wt \T|).
\end{equation}
Here, $\leapp$ means $\le C(\eps) D^{\eps^7}$.
The reason for \eqref{tube} is that: either in Thin case, we have $|\T|\leapp S^2E^{-2}|\B|D$; or in the Thick case, most tubes in $\T$ are contained in $\wt\T$ and each $2S\times D$-tube in $\wt \T$ contains at most $ 4S^2$ tubes in $\T$. 
\end{prop}

We will need a slight generalization which is our first lemma:
\begin{lemma}\label{cor21}
Fix a tiny $\eps > 0$. There exists a constant $C(\eps)$ with the following property: Let $\delta < 1$. Suppose that $\U$ is a set of $\delta \times 1$-rectangles in $[0, D]^2$. Let $S = D^{\eps/20}$, and define $T_r(\U)$ to be the set of $\delta \times D$-rectangles that contain at least $r$ rectangles from $\U$, $\wt T_{\tilde r}(\U)$ to be the set of $2S\delta \times D$-rectangles that contain at least $\td r$ rectangles from $\U$. Here we set $\tR$ to be a number $ \geapp Sr$. Then:
\begin{equation}\label{estimatetube}
    |T_r(\U)| \leapp_\e S^2 (r^{-2} |\U| D + |\wt T_{\tilde r}(\U)|).
\end{equation}
Here, $\leapp$ means $\le C(\eps) D^{\eps^7}$.
\end{lemma}

Note that Proposition \ref{twopointone} corresponds to $\delta = 1$.

\begin{proof}
Consider about $\de^{-1}$ many $\de$-separated directions. For each direction, we tile $[0,D]^2$ with rectangles pointing in this direction of dimensions $D\de\times D$. We call these rectangles cells. Denote these cells by $\{R_j\}_{j=1}^M$, then one actually sees that the number of cells is $M\sim \de^{-2}$. One also sees that these cells are essentially distinct.

Next, for each $\de\times 1$-rectangle  $U\in\U$, we will attach it to a cell. We observe that there is a cell $R_j$ such that all $\de\times D$-rectangles that contain this $\de\times 1$-rectangle $U$ are essentially contained in $R_j$. We attach this $U$ to $R_j$. Now for each $j$, we let $\U_j$ be the rectangles in $\U$ that are attached to $R_j$. We have
$$ \sum_j |\U_j|=|\U|, $$
\begin{equation}\label{tube1}
    |T_r(\U)|= \sum_j |T_r(\U_j)|, 
\end{equation} 
\begin{equation}\label{tube2}
     |\wt T_{\tilde r}(\U)|\gtrsim \sum_j |\wt T_{\tilde r}(\U_j)|.
\end{equation}
The reason for the last inequality is that a $2S\de\times D$-rectangle cannot contain $\de\times 1$-rectangles from more than $O(1)$ different $\U_j$, which means each tube in $\wt T_{\tilde r}(\U)$ belongs to $O(1)$ many $\wt T_{\tilde r}(\U_j)$.

For each $R_j$, we rescale so that $R_j$ becomes $[0,D]^2$. Also, $\U_j$ becomes a set of unit squares and any $\de\times D$-tube in $R_j$ becomes a $1\times D$-tube. Applying Proposition \ref{twopointone}, we see from \eqref{tube} that
$$ |T_r(\U_j)|\lessapprox S^2 (r^{-2}|\U_j|D+|\wt T_{\tilde r}(\U_j)|). $$
Summing over $j$ and using \eqref{tube1} and \eqref{tube2}, we proved \eqref{estimatetube}.
\end{proof}

Our second lemma concerns about the case when $X \sim \delta^{-1}$ in Theorem \ref{main}. It is actually the dual version of Corollary 5.5 in \cite{demeter2020small}. We state our lemma:

\begin{lemma}\label{highx}
Let $1 \le W \le \delta^{-1}$. Tile $[0, 1]^2$ with $W^{-1} \times \delta$ rectangles. Let $\B$ be a set of $\delta$-balls with at most one ball in each rectangle. 
We denote $|\B_{\max}| = W\delta^{-1}$.
Then for $r > \max(\delta^{1-3\e} |\B_{\max}|, 1)$, the number of $r$-rich tubes is bounded by 
\begin{equation*}
    |T_r(\B)| \lesssim_\e \delta^{-\e} \frac{|\B| |\B_{\max}|}{Wr^2}.
\end{equation*}
\end{lemma}

\begin{remark}
Lemma \ref{highx} actually takes care of the case when $r > W$ by rescaling.
\end{remark} 

To prove Lemma \ref{highx}, we need the following  dual version of Theorem 5.4 from \cite{demeter2020small}.

\begin{prop}\label{thm5.4dual}
Let $1 \le W \le \delta^{-1}$. Tile $[0, 1]^2$ with $W^{-1} \times \delta$-rectangles. Let $\B$ be a set of $\delta$-balls with at most $N$ balls in each rectangle. Let $r\ge 1$ and $T_r(\B)$ be a set of essentially distinct $\de$-tubes, each of which contains at least $r$ balls in $\B$. Then there exist a scale $1 \le s \le \delta^{-1}$ and an integer $M_s$ such that
\begin{equation} \label{thm5.4-1}
    |T_r(\B)| \lessapprox \frac{|\B| M_s \de^{-1}}{s r^2},
\end{equation}
\begin{equation} \label{thm5.4-2}
    r \lessapprox \frac{M_s \de^{-1}}{s^2},
\end{equation}
\begin{equation}\label{thm5.4-3}
    M_s \lessapprox Ns \max(1, s W \de).
\end{equation}
Here $\lessapprox$ means $\le C_\e \de^{-\e}$ for any $\e>0$.
\end{prop}

Let us quickly see how Proposition \ref{thm5.4dual} implies Lemma \ref{highx}.
\begin{proof}[Proof of Lemma \ref{highx}]
Apply Proposition \ref{thm5.4dual} with $N = 1$ to get a scale $s$ and an integer $M_s$. We claim that $sW\de \le 1$. If this is not true, then from \eqref{thm5.4-2} and \eqref{thm5.4-3}, we get
\begin{equation*}
    \de^{-3\e}W\le r \le C_\e\de^{-\e} \frac{M_s\de^{-1}}{s^2}\le C_\e^2 \de^{-2\e} s^2 W\delta \cdot \frac{\delta^{-1}}{s^2} =C_\e^2 \de^{-2\e} W,
\end{equation*}
which is a contradiction when $\de$ is small. Hence, $sW\de \le 1$, and so $M_s \leapp s$ and 
$$|T_r(\B)| \leapp \frac{|\B| \de^{-1}}{r^2} = \frac{|\B| |\B_{\max}|}{W r^2}$$
\end{proof}

Now, it suffices to prove Proposition \ref{thm5.4dual}.
\begin{proof}[Proof of Proposition \ref{thm5.4dual}]
It's convenient to explicitly write down \eqref{thm5.4-1}, \eqref{thm5.4-2} and \eqref{thm5.4-3} as

\begin{equation} \label{ineq1}
    |T_r(\B)| \le C_\e \de^{-\e} \frac{|\B| M_s \de^{-1}}{s r^2},
\end{equation}
\begin{equation} \label{ineq2}
    r \le C_\e \de^{-\e} \frac{M_s \de^{-1}}{s^2},
\end{equation}
\begin{equation}\label{ineq3}
    M_s \le C_\e \de^{-\e} Ns \max(1, s W \de).
\end{equation}

We induct on $\delta$ and $r$. There are three base cases.
\begin{itemize}
    \item $\delta \sim 1$,
    \item $r = 10 \delta^{-1}$,
    \item $NW \ge \delta^{-1+\eps/2}$.
\end{itemize}

The base case $\delta \sim 1$ is true by choosing large constant. The base case $r = 10 \delta^{-1}$ is taken care of by setting $s = 1$ and $M_s = 1$, and note that $|T_r(\B)| = 0$ since a $\delta$-tube contains at most $\delta^{-1}$ many $\de$-balls. For the base case $NW \ge \delta^{-1+\eps/2}$, set $s = \delta^{-1}$ and $M_s = s^2$. Then $r^2 |T_r(\B)|$ counts the number of triples $(B_1, B_2, T)\in \B\times\B\times T_r(\B)$ such that $B_1\cap T$ and $ B_2 \cap T$ are nonempty. For a given $B_1$, there are at most $\delta^{-1}$ many choices for $B_2$ and $\delta^{-1}$ many choices for $T$, hence $r^2 |T_r(\B)| \le |\B| \delta^{-2}$, which is what we want.

For the inductive step, assuming that the proposition holds for the tuple $\{(r,\de):r\ge 2\tilde r,\de=\td\de\}$ and $\{(r,\de):\de \ge 2\tilde \delta\}$, we prove the proposition for $r=\td r, \de=\td\de$. Let $\T\subset T_r(\B)$ be the subset of $\de$-tubes intersecting $\sim r$ balls of $\B$. If $|T_r(\B)|\ge 10 |\T|$, then $ |T_r(\B)|\le \frac{10}{9}|\T_{2r}(\B)| $. Using induction hypothesis to $\de$ and $ 2r$, we find $s$ and $M_s$ such that
\begin{equation*} 
    |T_{2r}(\B)| \le C_\e \de^{-\e} \frac{|\B| M_s \de^{-1}}{s (2r)^2}\Longrightarrow |T_{r}(\B)| \le C_\e \de^{-\e} \frac{|\B| M_s \de^{-1}}{s r^2},
\end{equation*}
\begin{equation*} 
    2r \le C_\e \de^{-\e} \frac{M_s \de^{-1}}{s^2},
\end{equation*}
\begin{equation*}
    M_s \le C_\e \de^{-\e} Ns \max(1, s W \de),
\end{equation*}
which verifies \eqref{ineq1}, \eqref{ineq2} and \eqref{ineq3}.
Hence we assume $|T_r(\B)|\le 10|\T|$.

We apply the rescaled version of Proposition \ref{twopointone} to $\B$ and $\T$. Note that $D=\de^{-1}$, $S=\de^{-\e/20}$. There are two possible cases. We discuss the two cases.

If we are in the thin case, we pick $s = 1, M_s = 1$ and obtain
\begin{gather*}
    |T_r(\B)|\le 10|\T| \le 10 C(\e)\de^{-\e^7}\de^{-\e/10} \frac{|\B| \delta^{-1}}{r^2} \le C_\e \de^{-\e} \frac{|\B| M_s \de^{-1}}{s r^2}\  (C_\e\textup{~large enough}), \\
    r \le 10 \delta^{-1}\le C_\e \de^{-\e} \frac{M_s \de^{-1}}{s^2},
\end{gather*}
which verifies \eqref{ineq1} and \eqref{ineq2}. Also, \eqref{ineq3} is easily verified.

If we are in the thick case, we obtain a set $\tT$ of $\gtrapprox Sr$-rich $S\delta$-tubes that contain $\geapp 1$ of the tubes in $\T$ (these two $``\gtrapprox"$ means $``\ge C(\e)^{-1}\de^{\e^7}"$), which implies
\begin{equation}\label{e1}
    |\T| \le C(\e)\de^{-\e^7} S^2 |\tilde \T|.
\end{equation}

Now we cover the balls in $\B$ using essentially distinct $S\delta$-balls denoted by $\tilde\B$. There is a partition
$$ \tilde\B=\bigsqcup_{M\textup{~dyadic}} \tilde \B_M, $$
where $\tilde\B_M$ are those $S\delta$-balls that contain $\sim M$ balls in $\B$.

We know each $\tilde T\in\tilde \T$ contains $\ge C(\e)^{-1}\de^{\e^7} Sr$ balls in $\B$. For a fixed $\td T$, by dyadic pigeonholing, there exist a dyadic number $M_{\td T}$ such that $\tilde T$ contains $\ge  C'(\e)^{-1}\de^{2\e^7} Sr$ balls in $\td\B_{M_{\td T}}$. By a further dyadic pigeonholing, there exists a dyadic $M$ such that a $C'(\e)^{-1} \de^{\e^7}$-fraction of tubes $\td T$ in $\tilde \T$ satisfy $M_{\td T}=M$, i.e., each of these $\td T$ contains $\ge  C'(\e)^{-1}\de^{2\e^7} Sr$ balls in $\td\B_M$. Now we fix $M$, and just still denote these $S\de$-tubes by $\tilde \T$. From \eqref{e1}, we have
\begin{equation}\label{e2}
   |\T| \le C''(\e)\de^{-2\e^7} S^2 |\tilde \T|. 
\end{equation}

A tube in $\tT$ contains more than $\tR = M^{-1} C'(\e)^{-1}\de^{2\e^7} Sr$ balls in $\tilde\B_M$. Furthermore, a $W^{-1} \times S\delta$ rectangle now contains at most $\tN = M^{-1} NS$ $S\delta$-balls in $\tilde\B_M$ since each $W^{-1} \times S\delta$ contains $S$ many $W^{-1} \times \delta$ rectangles.

Since we are not in the base cases, we assume $NW\le \de^{-1+\e/2}$ which implies $W\le (S\de)^{-1}$ (recall $S=\de^{-\e/20}$).
We can apply the induction hypothesis to 
\begin{equation}\label{induction}
    \tR = M^{-1} C'(\e)^{-1}\de^{2\e^7} Sr,\  \tW = W,\  \tde = S\delta,\ \tN = M^{-1} NS
\end{equation}
and the set of $\tilde \de$-balls $\tilde\B_M$. 
Thus, there exists $1 \le \tilde s \le (S\delta)^{-1}$ and $\td M_s$ such that
\begin{gather}
    |\tilde \T|\le |T_{\tilde r}(\tilde \B_M)| \le
    C_\e \td\de^{-\e} \frac{|\tilde\B_M| \tilde M_s \td\de^{-1}}{\tilde s {\tilde r}^2} \\
    \label{e3}\tR \le
    C_\e \td\de^{-\e} \frac{\tM_s\tde^{-1}}{\ts^2} , \\
    \label{e4}\tM_s \le
    C_\e \td\de^{-\e} \td N\ts \max(1, \ts W\tde).
\end{gather}
Now set $s = S \ts$ and $M_s = M \tM_s$. Combined with \eqref{e2}, we get
\begin{align*}
    |T_r(\B)|\le 10|\T|&\le C''(\e)\de^{-2\e^7}S^2|\td\T| \\
    &\le C''(\e)\de^{-2\e^7}S^2 C_\e \td\de^{-\e} \frac{|\tilde\B_M| \tilde M_s \td\de^{-1}}{\tilde s {\tilde r}^2}\\
    & = \big( C''(\e)C'(\e)^2\de^{-6\e^7}S^{-\e} \big)C_\e \de^{-\e}\frac{M|\td\B_M| M_s \de^{-1}}{s r^2}.
\end{align*}
Recall $S=\de^{-\e/20}$, so when $\de$ is small enough, $\big( C''(\e)C'(\e)^2\de^{-6\e^7}S^{-\e} \big)\le \frac{1}{10}$. Also, from the definition of $\td\B_M$, we have $M|\td\B_M|\le 2|\B|. $ Thus we have
$$ |T_r(\B)|\le C_\e \de^{-\e}\frac{|\B| M_s \de^{-1}}{s r^2}, $$
which closes the induction for \eqref{ineq1}.

From \eqref{induction} and \eqref{e3}, we have
\begin{align*}
    r &= C'(\e)\de^{-2\e^7} MS^{-1} \tR \\
    &\le C'(\e)\de^{-2\e^7} MS^{-1}  C_\e \td\de^{-\e} \frac{\tM_s\tde^{-1}}{\ts^2}\\
    &= \big( C'(\e)\de^{-2\e^7}S^{-\e} \big)C_\e\de^{-\e}\frac{M_s\de^{-1}}{s^2}\\
    &\le C_\e\de^{-\e}\frac{M_s\de^{-1}}{s^2}
\end{align*}
when $\de$ is small enough. This close the induction for \eqref{ineq2}.

From \eqref{e4}, we have
\begin{align*}
    M_s=M\td M_s &\le M C_\e \td\de^{-\e} \td N\ts \max(1, \ts W\tde)\\
    &=  C_\e (S\de)^{-\e} Ns \max(1, s W \de) \\ 
    &\le  C_\e \de^{-\e} Ns \max(1, s W \de),
\end{align*}
which close the induction for \eqref{ineq3}.

This completes the inductive step and thus finishes the proof.
\end{proof}

\begin{remark}
It is not clear to us whether Proposition \ref{thm5.4dual} follows from \cite{demeter2020small} Theorem 5.4. Actually, we only know Proposition \ref{thm5.4dual} implies \cite{demeter2020small} Theorem 5.4. 

Let us try to prove ``Theorem 5.4 in \cite{demeter2020small} $\Rightarrow$ Proposition \ref{thm5.4dual}", and see where is the gap.
Suppose we are given a set of $\de$-balls $\B$ satisfying the spacing condition in Proposition \ref{thm5.4dual}, and we want to estimate $|T_r(\B)|$. Following the notation in Section \ref{dualitysec}, we assume that $\B$ and $T_r(\B)$ lie in $\Pi_2$. We want to apply the tube-ball duality as in Section \ref{tubeball}, so that $\B\rightarrow l_1(\B)$ becomes $\de$-tubes in $\Pi_1$ and $T_r(\B)\rightarrow l_2^{-1}(T_r(\B))$ becomes $\de$-balls in $\Pi_1$. Note that $l_1(\B)$ inherits the spacing condition from $\B$ which meets the requirement in \cite{demeter2020small} Theorem 5.4. Also since the incidence are preserved, we have $l_2^{-1}(T_r(\B))$ is just $B_r(l_1(\B))$, the $r$-rich balls with respect to $l_1(\B)$. It seems we can use \cite{demeter2020small} Theorem 5.4, but the only shortage of this argument is that: $l_2^{-1}$ is not defined on all the tubes in $T_r(\B)$, but only defined on $T_r(\B)\cap \T_2$ (recall the definition of $\T_2$ in Section \ref{tubeball}). So, Theorem 5.4 in \cite{demeter2020small} actually implies the estimate
$$ |T_r(\B)\cap \T_2|\lessapprox \frac{|\B|M_s\de^{-1}}{s r^2},  $$
instead of \eqref{thm5.4-1}. What if we use a set of rotations $\{\rho_k\}_{k\le 100}$ as in Section \ref{subsec3.3} to partition the $r$-rich tubes $T_r(\B)=\cup_{k=1}^{100} T_r(\B)\cap \rho_k(\T_2)$, and estimate each $T_r(\B)\cap \rho_k(\T_2)$ independently? We see that $\rho_k^{-1}\big(T_r(\B)\cap \rho_k(\T_2)\big)=T_r(\rho_k^{-1}(\B))\cap\T_2$. Now we can apply duality so that the question becomes to estimate the $r$-rich $\de$-balls with respect to $l_1\big(\rho_k^{-1}(\B)\big)$. Let us explain the trouble. The original $\B$ are arranged in $W^{-1}\times \de$-rectangles whose edges are parallel to the axes, so $l_1(\B)$ satisfies the spacing condition in Theorem 5.4 \cite{demeter2020small}; but after rotation, $\rho_k^{-1}(\B)$ are arranged in tilted rectangles, so the dual tubes $l_1(\rho_k^{-1}(\B))$ no longer satisfies the spacing condition in Theorem 5.4 \cite{demeter2020small}. We remark that when $\rho_k$ is a $90^\circ$-rotation, then $\rho_k^{-1}(\B)$ are arranged in $\de\times W^{-1}$-rectangles whose shortest edges are now parallel to $x$-axis (originally were parallel to $y$-axis). In this case, $l_1(\rho_k^{-1}(\B))$ satisfies some spacing condition similar to Theorem \ref{main3} with $X=\de^{-1}$.

\end{remark}

\subsection{Proof of Theorem \ref{main}}
In this subsection, we prove Theorem \ref{main}. Let us recall Theorem \ref{main} here.

\begin{theorem}\label{mainn}
Let $1 \le W \le X \le \delta^{-1}$. Divide $[0, 1]^2$ into $W^{-1} \times X^{-1}$ rectangles as in Figure \ref{fig:generalcase}. Let $\B$ be a set of $\delta$-balls with at most one ball in each rectangle. 

We denote $|\B_{\max}|:=WX$ (as one can see that $\B$ contains at most $\sim WX$ balls).
Then for $r > \max(\delta^{1-2\e} |\B_{\max}|, 1)$, the number of $r$-rich tubes is bounded by
\begin{equation}\label{est}
    |T_r (\mathbb{B})| \le C_\e \delta^{-\e} |\B| |\B_{\max}| \cdot r^{-2} (r^{-1} + W^{-1}).
\end{equation}
\end{theorem}

The proof has the same idea as Theorem 4.1 in \cite{guth2019incidence}. There are three base cases. 
\begin{itemize}
    \item $r \ge 10 \delta^{-1}$
    \item $X \ge \delta^{-1+\eps/2}$
    \item $r \lesim \delta^{-\eps/4}$ or $W \lesim \delta^{-\eps/4}$
\end{itemize}
In the first base case $r \gesim \delta^{-1}$ we have $T_r (\B)$ is empty. The second base case $X \ge \delta^{-1+\eps/2}$ is dealt with by Lemma \ref{highx}. Actually, Lemma \ref{highx} (with $\e/2$ in place of $\e$) implies 
$$|T_r (\mathbb{B})| \le C_\e \delta^{-\e/2} |\B| W\de^{-1} \cdot r^{-2} W^{-1}\le C_\e \delta^{-\e} |\B| |\B_{\max}| \cdot r^{-2} (r^{-1} + W^{-1}),$$
where we use $|\B_{\textup{max}}|=WX\ge W\de^{-1+\e/2}$.

For the third base case $r \lesim \delta^{-\eps/4}$ or $W \lesim \delta^{-\eps/4}$, we use a double counting argument similar to \cite{ren2020incidence}. We count the number of triples $(B_1,B_2,T)\in \B\times\B\times T_r(\B)$ such that $B_1\cap T$ and $B_2\cap T$ are nonempty. Fix a ball $B_1\in \B$. For any dyadic radius $w$ $(X^{-1}\le w\le 1) $, consider those balls $B_2$ that are at distance $\sim w$ from $B_1$. The number of those $B_2$ is $\lesim wX(1 + wW)$. Also note that for two balls $B_1, B_2$ with distance $\sim w$, there are $\lesim \frac{1}{w}$ many tubes $T$ intersecting them. Thus, the number of triples is 
\begin{equation}
    \lesim |\B|\sum_{\begin{subarray}{c}X^{-1}\le w\le 1\\w\textup{~dyadic}\end{subarray}}wX(1+wW)\frac{1}{w}\lesssim (\log X)|\B| W X.
\end{equation}

Also, the number of triples has a lower bound $r^2 |T_r(\B)|$. Combining these two bounds, we get
\begin{equation*}
     |T_r(\B)| \lesssim (\log X)\frac{|\B| WX}{r^2} =(\log X) \frac{|\B| |\B_{\max}|}{r^2}\le C_\e \de^{-\e}  \frac{|\B| |\B_{\max}|}{r^2}(r^{-1}+W^{-1}).
\end{equation*}
which gives the estimate \eqref{est}.

For the inductive step, assuming that the theorem holds for the tuple $\{(r,\de):r\ge 2\tilde r,\de=\td\de\}$ and $\{(r,\de):\de \ge 2\tilde \delta\}$, we prove the theorem for $r=\td r, \de=\td\de$. In the rest of the proof, $``\leapp"$ will mean $``\le C(\e)\delta^{-O(\eps^7)}$". 

From the base case, we have $W\gtrsim \de^{-\e/4}$. We define $D = \delta^{-\eps^4}$, then $1 \le D \le W$. 
Cover the unit square with finitely overlapping $D^{-1}$-squares $\cQ=\{Q\}$. Let $\T\subset T_r(\B)$ be the set of tubes intersecting $\sim r$ balls from $\B$, and by induction we just need to consider the case \begin{equation}\label{f-1}
    |T_r(\B)|\le 10 |\T|,
\end{equation} 
as we did in the proof of Proposition \ref{thm5.4dual}.

A tube $T\in\T$ intersects $Q\in\cQ$ in a tube segment $U_D$ of dimensions $\delta \times D^{-1}$. Note that one $U_D$ can be essentially contained in many tubes $T \in \T$.
For dyadic $1 \le M \le \delta^{-1}$, let $\U_M$ be the set of essentially distinct tube segments $U_D$ which essentially contain $\sim M$ balls of $\B$. Then we have \begin{equation}\label{f0}
    \sum_{M \text{ dyadic}} MI(\U_M, \T) \sim I(\B, \T).
\end{equation}
Here $I(\U_M, \T)$ is the number of tuples $(U,T)\in \U_M\times \T$ such that $U$ is essentially contained in $T$.
We may choose a dyadic $M$ satisfying
\begin{equation}\label{f1}
    M I(\U_M, \T) \geapp I(\B, \T).
\end{equation}
Next, let $\T_E\subset \T$ be the set of tubes that contain $\sim E$ tube segments $U_D\in\U_M$. Since $\sum_{E \text{ dyadic}} I(\U_M, \T_E) \geapp I(\U_M, \T)$, we may choose $E$  satisfying
\begin{equation}\label{f2}
    MI(\U_M, \T_E) \geapp I(\B, \T).
\end{equation}

Note that we have $I(\B,\T_E)\gtrsim MI(\U_M,\T_E)$, together with \eqref{f2} to obtain $I(\B,\T_E)\geapp I(\B, \T)$. Since each tube in $\T$ contains $\sim r$ balls of $\B$ by definition, we get 
\begin{equation}\label{f3}
    |\T_E| \geapp |\T|.
\end{equation}

Since each $T\in\T_E$ contains $\sim E$ tube segments $U_D\in\U_M$ and each $U_D\in \U_M$ contains $\sim M$ balls in $\B$, we have each $T\in\T_E$ contain $\gtrsim ME$ balls in $\B$. On the other hand, every tube in $\T$ contains $\sim r$ balls in $\B$ by definition. So, we have
\begin{equation}\label{f4}
    r\gtrsim ME.
\end{equation}
Also note that from \eqref{f2}, we have
\begin{equation*}
    r |\T| \sim I(\B, \T) \leapp M I(\U_M, \T_E) \leapp ME |\T_E|\le ME |\T|,
\end{equation*}
which implies
\begin{equation}\label{f5}
    r\lessapprox ME.
\end{equation}

Now we apply a rescaled version of Lemma \ref{cor21} with $\U=\U_M$ and $r=E$ to get
\begin{equation}
    |\T_E| \lessapprox S^2 (E^{-2} |\U_M| D + |\wt T_{\td r}(\U_M)|):=\uppercase\expandafter{\romannumeral1}+
\uppercase\expandafter{\romannumeral2}.
\end{equation}
Here, $S=D^{\e/20}$, $\td r\gtrapprox SE$. $\wt T_{\td r}(\U_M)$ is the set of $2 S\de\times 1$-tubes that contain at least $\td r$ rectangles from $\U_M$.

We would like to rewrite the second term a little bit. Note that by definition each $U_D\in\U_M$ contain $\sim M$ balls in $\B$, so $\wt T_{\td r}(\U_M)\subset \wt T_{\td r M}(\B)$. Here $\wt T_{\td r M}(\B)$ is the set of $2 S\de\times 1$-tubes that contain at least $\td r M$ balls from $\B$. We have
\begin{equation}\label{mainest}
    |\T_E| \lessapprox S^2 (E^{-2} |\U_M| D + |\wt T_{r_1 }(\B)|):=\uppercase\expandafter{\romannumeral1}+
\uppercase\expandafter{\romannumeral2},
\end{equation}
where $r_1=\td r M\gtrapprox SEM\gtrapprox Sr$.

\subsubsection{Estimate of \uppercase\expandafter{\romannumeral1}}

Fix a $D^{-1}$-square $Q\in\cQ$ in our finitely overlapping covering. Also, we consider the set of tubes $\U_Q:=\{U\in\U_M: U\subset Q\}$ and the set of balls $\B_Q:=\{B\in\B: B\subset Q\}$. If we rescale $Q$ to $[0,1]^2$, then $\U_Q$ becomes a set of $D\de$-tubes and $\B_Q$ becomes a set of $D\de $-balls. Meanwhile, $\B_Q$ satisfies the spacing condition in Theorem \ref{mainn}.
We use the induction hypothesis of Theorem \ref{mainn} to $\B_Q$ with
\begin{enumerate}
    \item $\de' = D\delta$,
    
    \item $r' = M$,
    
    \item $W' = W/D$, $X' = X/D$.
\end{enumerate}
In order to apply the induction hypothesis, we need to check $r'=M >\max(\de'^{1-2\e}W'X',1)$. Actually, by \eqref{f5} and noting $E\le D, r>\delta^{1-2\e} WX, D=\de^{-\e^4}$, we have
\begin{equation*}
    M \ge C(\e)^{-1}\de^{\e^7} E^{-1} r \ge C(\e)^{-1}\de^{\e^7} D^{-1} \delta^{1-2\e} WX \ge  \de'^{1-2\e} W' X'.
\end{equation*}
To check $M \ge 2$, by the base case $r \ge \delta^{-\eps/4}$ is big, and $E\le D = \delta^{-\eps^4}$ is small, so \eqref{f5} implies $M\ge 2$.

Now we can apply induction. From \eqref{est}, we obtain:
$$|\U_Q|\le C_\e (D\de)^{-\e} |\B_Q|D^{-2}W X\cdot M^{-2}(M^{-1}+D W^{-1}).$$
Summing over $Q\in\cQ$ we get
\begin{align*}
    \uppercase\expandafter{\romannumeral1}=S^2 E^{-2} D |\U_M|&=S^2 E^{-2} D \sum_{Q}|\U_Q|\\
    &\le S^2 E^{-2} D \cdot \sum_Q C_\e (D\de)^{-\e} |\B_Q|D^{-2}W X\cdot M^{-2}(M^{-1}+D W^{-1})\\
    &=S^2 D^{-\e} C_\e \de^{-\e} \sum_Q |\B_Q| WX (ME)^{-2} ((MD)^{-1}+W^{-1}) 
\end{align*}
Since $S=D^{\e/20}$, $\sum_Q|\B_Q|=|\B|$, $ME\gtrapprox r$, $E\le D$, we have
\begin{equation}\label{term1}
    \uppercase\expandafter{\romannumeral1}\lessapprox  D^{-\e/2} C_\e \de^{-\e} |\B| WX r^{-2} (r^{-1}+W^{-1}).
\end{equation} 
Recall $``\lessapprox"$ means $\le C(\e) \de^{-\e^7}$.

\subsubsection{Estimate of \uppercase\expandafter{\romannumeral2}}
For the second term in \eqref{mainest}, we have a collection of $2S\delta$-tubes $\wt T_{r_1}(\B)$, each of which intersects $r_1 \geapp Sr$ balls of $\B$. Let $\wt \B$ be the set of balls formed by thickening each $\delta$-ball of $\B$ to a $S\delta$-ball. From the base case, we have $X\le (S\de)^{-1}$, so $\wt \B$ is a set of $S\de$-separated balls satisfying the spacing condition in Theorem \ref{mainn}.

Apply the induction hypothesis to $\wt \B$ with $\delta' = S\delta$, $r'=r_1\gtrapprox Sr$, $W' = W$, $X' = X$ (it is easy to check $r'> \max(\de'^{1-2\e}W' X',1)$.
We obtain from \eqref{est} 
$$ |\wt T_{r_1}(\wt\B)|\le C_\e (S\de)^{-\e} |\B|WX\cdot r_1^{-2}(r_1^{-1}+W^{-1}). $$
So, we have
\begin{align}
\nonumber\uppercase\expandafter{\romannumeral2}=S^2 |\wt T_{r_1}(\B)| &\le S^2 C_\e (S\de)^{-\e} |\B|WX\cdot r_1^{-2}(r_1^{-1}+W^{-1})\\
\nonumber&\lessapprox S^2 C_\e (S\de)^{-\e} |\B|WX\cdot (Sr)^{-2} ((Sr)^{-1} + W^{-1})  \\
\label{term2}&\le S^{-\e}C_\e \de^{-\e} |\B|WX\cdot r^{-2} (r^{-1} + W^{-1}) 
\end{align}
Combining \eqref{f-1}, \eqref{f3}, \eqref{term1}, \eqref{term2}, and recalling $D=\de^{-\e^4}, S=D^{\e/20}$, we have
$$ |T_r(\B)|\lessapprox |\T_E|\lessapprox \de^{\e^6/20}C_\e \de^{-\e} |\B|WX\cdot r^{-2} (r^{-1} + W^{-1}).  $$
Recall $``\lessapprox"$ means $\le C(\e) \de^{-O(\e^7)}$. We see if $\de$ is small enough, this closes the induction for \eqref{est}.

\section{Proof of Theorem \ref{thmfur}} \label{fursec}
In this section we prove Theorem \ref{thmfur} which we recall here.

\begin{theorem}\label{fur}
Let $1 \le W \le X \le \delta^{-1}$. Let $\T$ be a collection of essentially distinct $\delta$-tubes in $[0,1]^2$ that satisfies the following spacing condition: every $W^{-1}$-tube contains at most $\frac{X}{W}$ many tubes of $\T$, and the directions of these tubes are $\frac{1}{X}$-separated. We also assume $|\T|\sim XW$.

Assume for each $T$ there is a set of $\de$-balls $Y(T)=\{B_\de\}$ satisfying: each ball in $Y(T)$ intersects $T$; $\#Y(T)\sim \de^{-\alpha}$ and the balls in $Y(T)$ have spacing $\gtrsim \de^{\alpha}$. 
Define the union of these $\de$-balls to be $\B=\cup_T Y(T)$.
Then we have the estimate
\begin{equation}
    |\B|\gtrsim (\log\de^{-1})^{3.5}\min(\de^{-\alpha-1},\de^{-\frac{3}{2}\alpha}(XW)^{\frac{1}{2}},\de^{-\alpha}XW).
\end{equation}
\end{theorem}

\begin{remark}
As we discussed in Section \ref{dualitysec}, it's more intuitive to view the tubes $\T$ in the theorem through their corresponding dual balls in the dual space. Actually, we see that the dual balls of $\T$ have the configuration as in Figure \ref{fig:generalcase}.
\end{remark}

\subsection{A try using incidence estimates}
It seems we can use Theorem \ref{main2} to study Theorem \ref{fur}. We discuss this here and will see where we fail.

Since the spacing condition for $\T$ is the same in Theorem \ref{main2} and Theorem \ref{fur}, we can use the incidence estimate in Theorem \ref{main2}. We denote by $I(\B,\T)$ the incidence between $\T$ and $\B$. Since each tube intersects $\ge \de^{-\alpha}$ balls in $\B$, we simply get the lower bound for incidence:
$$ I(\B,\T)\gtrsim \de^{-\alpha}|\T|. $$
For the upper bound, since $I(\B,\T)\lesim \sum_{r\textup{~dyadic}}r |B_r(\T)|$, there exists a dyadic $r$ such that
$$ I(\B,\T)\lessapprox r|B_r(\T)| $$
(In this subsection $``\lessapprox"$ means $``\le C_\e \de^{-\e}"$ for any $\e>0$).
So, we have
\begin{equation}\label{fur1}
   \de^{-\alpha}|\T|\lessapprox r|B_r(\T)|.
\end{equation}

Our estimates will be based on \eqref{fur1}. When $r\lessapprox \max(\de |\T|,1)$, we have $\de^{-\alpha}|\T|\lessapprox r|B_r(\T)|\lessapprox \max(\de|\T|,1) |\B|$, which implies that 
\begin{equation}\label{fur2}
    |\B|\gtrapprox \min(\de^{-\alpha-1},\de^{-\alpha}XW).
\end{equation}
When $\max(\de|\T|,1)\lesim r\le W$, by Theorem \ref{main2} we have
$|B_r(\T)|\lessapprox |\T|^2 r^{-3}.$
Combined with the trivial bound $|B_r(\T)|\le |\B|$, we get $\de^{-\alpha}|\T|\lessapprox r \min(|\T|^2r^{-3},|\B|)\le |\T|^{2/3}|\B|^{2/3}$, and hence
\begin{equation}\label{fur3}
    |\B|\gtrapprox \de^{-\frac{3}{2}\alpha}|\T|^{1/2}\gtrsim \de^{-\frac{3}{2}\alpha}(XW)^{1/2}.
\end{equation}
When $r\ge W$, by Theorem \ref{main2} we have
$|B_r(\T)|\lessapprox |\T|^2 r^{-2}W^{-1}.$
Combined with the trivial bound $|B_r(\T)|\le |\B|$, we get $\de^{-\alpha}|\T|\lessapprox r \min(|\T|^2r^{-2}W^{-1},|\B|)\le |\T||\B|^{1/2}W^{-1/2}$, and hence
\begin{equation}\label{fur4}
    |\B|\gtrapprox \de^{-2\alpha}W.
\end{equation}
Combining \eqref{fur2}, \eqref{fur3} and \eqref{fur4}, we obtain
\begin{equation}
    |\B|\gtrapprox \min (\de^{-\alpha-1},\de^{-\frac{3}{2}\alpha}(XW)^{1/2}, \de^{-\alpha}XW, \de^{-2\alpha}W).
\end{equation}
In the above inequality, we see that the fourth term $\de^{-2\alpha}W$ is not good in the case when $W=1$ and $X=\de^{-1}$. This is our main enemy, which is exactly the case of Question \ref{question}.

\subsection{The proof of Theorem \ref{fur}}
We will use a graph-theoretic proof whose idea dates back to \cite{szekely1997crossing}.
First we discuss our main tool: crossing number. For a graph $G$, the crossing number $cr(G)$ of $G$ is the lowest number of edge crossings of a plane drawing of the graph $G$. We have the following well-known result for crossing numbers. A detailed discussion can be found in \cite{guth2016polynomial}.
\begin{lemma}[Crossing number]\label{crossing}
For a graph $G$ with $n$ vertices and $e$ edges, we have
\begin{equation}
    n\gtrsim \min (e,\frac{e^{3/2}}{cr(G)^{1/2}}).
\end{equation}
\end{lemma}

To prove Theorem \ref{fur}, we do several reductions. First, we can assume all the $\de$-tubes from $\T$ lie in $[0,1]^2$ and each tube forms an angle $\le \frac{1}{10}$ with $y$-axis. We also assume $\de^{-1}$ is an integer, so we can partition $[0,1]^2$ into lattice $\de$-squares, denoted by $[0,1]^2=\sqcup Q$. Here, each $Q$ is a square with length $\de$ and center $(\de(n-\frac{1}{2}),\de (m-\frac{1}{2}))$ for some $1\le m,n\le \de^{-1}$. We denote the set of these squares by $\cQ_\de$. Since it is not harmful to replace the $\de$-balls by $\de$-cubes, we may assume $\B\subset \cQ_\de$.

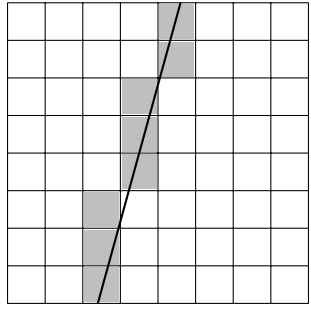
\begin{figure}
    \centering
    
\begin{tikzpicture}
\draw[step=0.5cm, black, thin] (0,0) grid (4,4);
\fill[gray!50] (1.02,0.02) rectangle (1.48,0.48);
\fill[gray!50] (1.02,0.52) rectangle (1.48,0.98);
\fill[gray!50] (1.02,1.02) rectangle (1.48,1.48);
\fill[gray!50] (1.52,1.52) rectangle (1.98,1.98);
\fill[gray!50] (1.52,2.02) rectangle (1.98,2.48);
\fill[gray!50] (1.52,2.52) rectangle (1.98,2.98);
\fill[gray!50] (2.02,3.02) rectangle (2.48,3.48);
\fill[gray!50] (2.02,3.52) rectangle (2.48,3.98);
\draw[black, thick] (1.2,0) -- (2.3,4);
\end{tikzpicture}

    \caption{Pseudo-tube}
    \label{sdtube}
\end{figure}

To make the proof clear, we need a substitution for tube, which we call \textit{pseudo-tube}. We give the definition of pseudo-tube. Given an $1\times \de$ tube $T$ which forms an angle $\le \frac{1}{10}$ with $y$-axis and lie in $[0,1]^2$,  we define its corresponding pseudo-tube $\td T$ as in Figure \ref{sdtube}.
Denote the core line of $T$ by $l$. The squares in $\cQ_\de$ form $\de^{-1}$ many rows. We see that $l$ intersect each row with at most two squares. For each row, if $l$ intersect this row with one square, we pick this square; if $l$ intersect this row with two squares, we pick the left square.
We define $\td T$ to be the union of these $\de^{-1}$ many squares we just picked. We call $\td T$ the corresponding pseudo-tube of $T$.

It's not hard to check that we can make the reduction so that the $\T$ in Theorem \ref{fur} is a set of pseudo-tubes and $Y(T)$ is a set of $\de$-squares contained in $\td T$. Without ambiguity, we still call pseudo-tube as tube and use $T$ instead of $\td T$.

\medskip

The next reduction is to guarantee some uniformity property among tubes. We set $Y'(T):=\{B\in\B:B\subset T\}$ (note that $Y(T)$ is a subset of $Y'(T)$). We label the squares in $Y'(T)$ one-by-one from bottom to top as $Y'(T)=\{Q_1,Q_2\cdots, Q_m\}$. Here $m=\#Y'(T)$ and the $y$-coordinates of $Q_i$ is less than that of $Q_{i+1}$. We define the distance between nearby squares as $d_i:=\dist(Q_i,Q_{i+1})$. Define the $d$-index set as $I_d:=\{i: d_i\sim d, 1\le i\le m-1\}$.

We claim that there exists a number $d\lesim \de^{\alpha}$
such that
\begin{equation}\label{ibig}
    |I_d|\gtrsim (\log\de^{-1})^{-1}d^{-1}.
\end{equation}
Recalling the condition of Theorem \ref{fur}, we have that each $Y(T)$ is a subset of $Y'(T)$ satisfying: $\#Y(T)\sim \de^{-\alpha}$ and each pair of nearby squares in $Y(T)$ have distance $\gtrsim \de^{\alpha}$. From this, we see that
\begin{equation}
    \sum_{d_i\lesim \de^{\alpha}}d_i\gtrsim 1.
\end{equation}
So, by pigeonhole principle we can find $d\lesim\de^{\alpha}$ such that
$$ 1\lesim \log(\de^{-1})\sum_{d_i\sim d}d_i\sim \log(\de^{-1}) d|I_d|, $$
which gives \eqref{ibig}.

For each $T\in\T$, there exists a $d_T\lesim \de^{\alpha}$ such that \eqref{ibig} holds for $d=d_T$. By dyadic pigeonholing, we choose a typical $d$ so that there is a set $\T'\subset \T$ such that $|\T'|\gtrsim (\log\de^{-1})^{-1}|\T|$ and $d_T=d$ for any $T\in\T'$. We denote $d=\de^\beta$, $\B'=\cup_{T\in\T'}Y'(T)$. Since $\B'\subset \B$ and $\alpha\le \beta$, we only need to prove:
\begin{equation}
    |\B'|\gtrsim (\log\de^{-1})^{3.5}\min(\de^{-\beta-1},\de^{-\frac{3}{2}\beta}(XW)^{\frac{1}{2}},\de^{-\beta}XW).
\end{equation}
If we abuse the notation and still write $\beta, \T',\B'$ as $\alpha,\T,\B$, we actually reduce Theorem \ref{fur} to the following problem.

\begin{theorem}\label{last}
Let $1 \le W \le X \le \delta^{-1}$. Let $\T$ be a collection of essentially distinct $\delta$-pseudo-tubes in $[0,1]^2$ that satisfies the following spacing condition: every $W^{-1}$-tube contains at most $\frac{X}{W}$ many tubes of $\T$, and the directions of these tubes are $\frac{1}{X}$-separated. We also assume $|\T|\gtrsim (\log\de^{-1})^{-1} XW$.

Let $\B=\{B_\de\}\subset \cQ_\de$ be a set of $\de$-squares and for each $T\in\T$ define $Y(T):=\{ B_{\de}\in \B: B_{\de}\subset T\}$. Suppose each $Y(T)$ satisfies \eqref{ibig} for $d=\de^\alpha$. Then we have the estimate
\begin{equation}
    |\B|\gtrsim (\log\de^{-1})^{3.5}\min(\de^{-\alpha-1},\de^{-\frac{3}{2}\alpha}(XW)^{\frac{1}{2}},\de^{-\alpha}XW).
\end{equation}
\end{theorem}

\begin{proof}[Proof of Theorem \ref{last}]
We construct a graph $G=(V,E)$ in the following way. Let the vertices $V$ be the centers of squares in $\B$. For a $T\in\T$, consider all the pairs of nearby squares in $Y(T)$. We link the centers of  each nearby squares with an edge. Let $E$ be the edges formed in this way for all $T\in\T$.

Each pair of tubes contribute at most one crossing (but they may share a lot of edges), so we have
\begin{equation}
    cr(G)\le |\T|^2.
\end{equation}
On the other hand by Lemma \ref{crossing} we have
\begin{equation}
     |\B|\gtrsim \min (|E|,\frac{|E|^{3/2}}{cr(G)^{1/2}}).
\end{equation}
So, we have
\begin{equation}
     |\B|\gtrsim \min (|E|,\frac{|E|^{3/2}}{|\T|}).
\end{equation}

We will discuss two cases. We remind the readers that $(\log\de^{-1})^{-1}XW\lesim |\T|\lesim XW$.

\textit{Case 1}: $XW\lesim \de^{-2+\alpha}$.

We prove that $|E|\gtrsim (\log\de^{-1})^{-2}\de^{-\alpha}|\T|$, so as a result we obtain
\begin{equation}\label{last1}
    |\B|\gtrsim (\log\de^{-1})^{-3.5} \min(\de^{-\alpha} XW,\de^{-\frac{3}{2}\alpha}(XW)^{1/2}).
\end{equation} 
For each edge $e\in E$, define $n_e$ to be the number of tubes $T\in\T$ that contain $e$. We have 
$$|E|=\sum_{e\in E}\ \sum_{T\in\T, e\subset T}\frac{1}{n_e}=\sum_{T\in \T}\ \sum_{e\in E, e\subset T}\frac{1}{n_e}.$$

It suffices to show for any fixed $T_0$,
$$ \sum_{e\in E, e\subset T_0}\frac{1}{n_e}\gtrsim (\log\de^{-1})^{-2}\de^{-\alpha}. $$
Recalling the condition for $Y(T)$ in Theorem \ref{last} and \eqref{ibig}, we have
$$ \#\{ e\subset T: \textup{length}(e)\sim \de^{\alpha} \} \gtrsim (\log\de^{-1})^{-1}\de^{-\alpha}. $$
So by Cauchy-Schwartz inequality, we have
$$ \sum_{e\in E, e\subset T_0}\frac{1}{n_e}\ge \sum_{e\in E, e\subset T_0, \textup{length}(e)\sim\de^{\alpha}}\frac{1}{n_e}\ge \frac{(\log\de^{-1})^{-2}\de^{-2\alpha}}{\sum_{e\in E, e\subset T_0,\textup{length}(e)\sim\de^{\alpha}}n_e}.$$

It suffices to prove
\begin{equation}\label{lastt2}
    \sum_{e\in E, e\subset T_0,\textup{length}(e)\sim\de^{\alpha}}n_e\lesim \de^{-\alpha}. 
\end{equation} 
For $e\in E, T\in\T$, we define $\chi(e,T)=1$ if $e\subset T$ and $=0$ otherwise.
We rewrite the left hand side above as
\begin{equation}\label{lastt1}
    \sum_{e\subset T_0, \textup{length}(e)\sim\de^{\alpha}}\sum_{T\in\T}\chi(e,T).
\end{equation}
Note that if $\chi(e,T)=1$ for some $e\subset T_0$ satisfying $\textup{length}(e)\sim \de^{\alpha}$, then the angle between $T_0$ and $T$ is less than $\de^{1-\alpha}$. We will analyze $T$ according to the angle $\mu=\angle(T_0,T)$. It's easy to see that those $T$ that form an angle $\sim\mu$ with $T_0$ lie in a $1\times \mu$ fat tube, and by the spacing condition of $\T$ we have
$$\#\{T:\angle(T_0,T)\sim\mu\}\lesim \mu^{2} XW\lesim \mu^{2}\de^{-2+\alpha}.$$
In the last inequality we use the assumption $XW\lesim \de^{-2+\alpha}$.

We further rewrite \eqref{lastt1}  as:
$$\sum_{\mu\lesim \de^{1-\alpha}}\ \sum_{T:\angle(T_0,T)\sim\mu}\ \sum_{e\subset T_0, \textup{length}(e)\sim\de^{\alpha}}\chi(e,T).$$
Note that when $\angle(T_0,T)\sim\mu$, we have $\sum_{e\subset T_0, \textup{length}(e)\sim\de^{\alpha}}\chi(e,T)\le 
\frac{\textup{length}(T_0\cap T)}{\de^\alpha}\lesim \mu^{-1}\de^{1-\alpha}$,
so the inequality above is less than
$$\sum_{\mu\lesim \de^{1-\alpha}}\ \sum_{T:\angle(T_0,T)\sim\mu}\mu^{-1}\de^{1-\alpha}\lesim \sum_{\mu\lesim \de^{1-\alpha}} \mu^{2}\de^{-2+\alpha}\mu^{-1}\de^{1-\alpha}\lesim \de^{-\alpha}.$$
This finishes the proof of \eqref{lastt2}.

\textit{Case 2}: $XW\gtrsim \de^{-2+\alpha}$.
In this case, we choose another pair $(X',W')$ so that $X'\le X, W'\le W$, $1\le W'\le X'\le \de^{-1}$ and $X'W'\sim \de^{-2+\alpha}$. We through away some tubes from $\T$ so that it satisfies the spacing condition for new parameters $(X',W')$. This is easily seen from the dual picture as in Figure \ref{fig:generalcase}. Originally, the balls are evenly spaced in the $W^{-1}\times X^{-1}$-grid. We throw away some balls so that it fits into the $W'^{-1}\times X'^{-1}$-grid. We apply Case 1 to the new parameter $(X',W')$ to obtain
\begin{equation}\label{last2}
    |\B|\gtrsim (\log\de^{-1})^{-3.5} \min(\de^{-\alpha} X'W',\de^{-\frac{3}{2}\alpha}(X'W')^{1/2})=(\log\de^{-1})^{-3.5} \min(\de^{-2},\de^{-\alpha-1}).
\end{equation} 
Combining \eqref{last1} and \eqref{last2}, we obtain the desired estimate
\[
    |\B|\gtrsim (\log\de^{-1})^{-3.5} \min(\de^{-\alpha-1},\de^{-\alpha} XW,\de^{-\frac{3}{2}\alpha}(XW)^{1/2}).
\]

\end{proof}


\bibliographystyle{abbrv}

\end{document}